\documentclass{amsart}

\usepackage{latexsym}
\usepackage[all]{xy}
\usepackage{amssymb,amsmath}
 \usepackage{verbatim}

\def\wt{\widetilde}

\def\bZ{\mathbb{Z}}

\def\bC{\mathbb{C}}
\def\bP{\mathbb{P}}

\def\cO{\mathcal{O}}

\def\W{\mathcal{W}}

\def\I{\mathcal{I}}

 \def\Syz{\mathrm{Syz}}

\def\Pic{\mathrm{Pic}}

\def\dim{\mathrm{dim}}

\def\ker{\mathrm{ker}}

\def\da{\dasharrow}

 \newtheorem{theorem}{Theorem}%[chapter]
\newtheorem{lemma}[theorem]{Lemma}
\newtheorem{prop}[theorem]{Proposition}

\newtheorem{corollary}[theorem]{Corollary}

\newtheorem{remark}[theorem]{Remark}

\begin{document}

\title{A characterization of bielliptic curves via syzygy schemes}

\author{Marian Aprodu}
\address{Faculty of Mathematics and Computer Science, University of Bucharest, 14 Academiei Street, 010014 Bucharest, Romania}
\address{Simion Stoilow Institute of Mathematics of the Romanian Academy, P.O. Box 1-764, 014700 Bucharest, Romania}
\email{marian.aprodu@imar.ro, marian.aprodu@fmi.unibuc.ro}
\author{Andrea Bruno}
\address{Dipartimento di Matematica e Fisica, Universit\`a degli Studi Roma Tre, Largo San Leonardo Murialdo, I-00146 Roma, Italy}
\email{bruno@mat.uniroma3.it}
\author{Edoardo Sernesi}
\address{Dipartimento di Matematica e Fisica, Universit\`a degli Studi Roma Tre, Largo San Leonardo Murialdo, I-00146 Roma, Italy}
\email{sernesi@mat.uniroma3.it}

\thanks{MA  thanks the University Roma Tre and Max Planck Institute Math. Bonn for hospitality during the preparation of this work. MA's work was partly supported by a grant of Ministery of Research and Innovation, CNCS - UEFISCDI, project number PN-III-P4-ID-PCE-2016-0030, within PNCDI~III. ES  thanks the Max Planck Institute Math. Bonn for hospitality during the preparation of this work. ES was supported by the PRIN Project ''Geometry of Algebraic Varieties'' and by INDAM-GNSAGA}
%\author{Marian Aprodu, Andrea Bruno, Edoardo Sernesi}
\date{\today}
\maketitle
%%%%%%%%%%%%%%%%%%%%%%%%%%%%%%%%%%%%%%%%%%%%%%%%%%

\begin{abstract}
We prove that a canonical curve $C$ of genus $\ge 11$ is bielliptic if and only if its second syzygy scheme $\mathrm{Syz}_2(C)$ is different from $C$.
\end{abstract}

\section*{Introduction}

Let $V$ be a complex vector space of dimension $(N+1)$ and let $\gamma$ be an element of $\wedge^pV\otimes V$ for some $1\le p\le N$. The \emph{syzygy scheme} associated to $\gamma$ is the largest subscheme $\mathrm{Syz}(\gamma)$ of $\mathbb P^N$ such that $\gamma$ represents a syzygy of $\mathrm{Syz}(\gamma)$, see \cite[\S 1]{mG82}. It means that $\delta(\gamma)\in\wedge^{p-1}V\otimes I_{\mathrm{Syz}(\gamma),2}$  where $\delta:\wedge^pV\otimes V\to \wedge^{p-1}V\otimes S_{V,2}$ denotes the Koszul differential. An element $\gamma$ gives rise to a space of quadrics on $\mathbb P^N$, the \emph{quadrics involved in $\gamma$}, defined as the image of the induced map $\wedge^{p-1}V^*\to S_{V,2}$.

Syzygy schemes are defined abstractly for any element $\gamma$ using multi-algebra tools. However, the most interesting cases occur when $\gamma$  already represents a Koszul class of a linearly--normal variety $X\subset \mathbb P^N$ since  $X$ is contained in $\mathrm{Syz}(\gamma)$. We use the simplified notation $K_{p,1}(X)$ for the Koszul cohomology spaces in the linear strand of the variety $X$ and we refer to \cite{mG84} and \cite{AN10} for the basics on Koszul cohomology.
In this context, one can define the \emph{p-th syzygy variety of $X$} as the scheme-theoretic intersection \cite{mG84}, \cite{sE94}, \cite{vB07}, \cite{ES12}:
\[
\mathrm{Syz}_p(X)=\bigcap_{0\ne\gamma\in K_{p,1}(X)}\mathrm{Syz}(\gamma).
\]

For large $p$, the syzygy schemes have been classified, \cite{mG84}, \cite{sE94}. For example, if $p>N-\mathrm{dim}(X)$, then $\mathrm{Syz}_p(X)=\mathbb P^N$ and for $p=N-\mathrm{dim}(X)$, $\mathrm{Syz}_p(X)$ is either $\mathbb P^N$ or it is a variety of minimal degree in which case $X=\mathrm{Syz}_p(X)$, \cite[Theorem (3.c.1)]{mG84}. The next two cases are also well understood \cite{mG84}, \cite{sE94}. The general expectation is that for large $p$ their degree is bounded in function of $p$ and $\mathrm{dim}(X)$.

In the curve case, the analysis of the last syzygy schemes leads to interesting conclusions. If $C\subset \mathbb P^{g-1}$ is a (nonhyperelliptic) canonical curve the above--quoted Theorem (3.c.1) of \cite{mG84} shows that $K_{g-3,1}(C)\ne 0$ if and only if $C$ is either trigonal or a nonsingular plane quintic.  By duality, \cite[Corollary (2.c.10)]{mG84}, it shows that the ideal of the canonical curve $C$ is generated by quadrics, unless $C$ is trigonal or a nonsingular plane quintic which is the classical Enriques--Petri Theorem. This result was the main source of inspiration for the celebrated Green conjecture which relates the Koszul cohomology of the canonical curves with the Clifford index \cite{mG84}, \cite{cV02}, \cite{cV05}. The next case of Green's conjecture, solved independently and almost simultaneously by C. Voisin and F.--O. Schreyer, shows that the ideal of a canonical curve $C$ is generated by quadrics and the relations among the quadrics are generated by linear relations, unless $C$ is 4-gonal or a plane sextic i.e. $\mathrm{Cliff}(C)=2$. The property that the ideal be generated by quadrics and the relations among the quadrics be generated by linear ones  is called \emph{property $(N_2)$}, whereas \emph{property $(N_1)$} means simply that the ideal is generated by quadrics, see~\cite{mG84}.

The results we mentioned above regard mostly the syzygy schemes that appear towards the end of the minimal resolution. Very little is known for the syzygy schemes connected to the \emph{beginning} of the resolution, the only place where they have appeared being, to the best of our knowledge \cite{CEFS13} where paracanonical curves of genus 8 were investigated. In our paper, we show that $\mathrm{Syz}_2$ has a strong influence on the geometry of curves. Our results can be summarized into the following

\begin{theorem}
\label{thm:mainTHM}
Let $C\subset \mathbb P^{g-1}$ be a nonhyperelliptic canonical curve of genus $\ge 11$ and gonality at least $4$. Then $\mathrm{Syz}_2(C)\ne C$ if and only if $C$ is not bielliptic. In the bielliptic case, $\mathrm{Syz}_2(C)$ is an elliptic cone.
\end{theorem}

%In the bielliptic case, we prove that the second syzygy variety is an elliptic cone. 

The second syzygy schemes can be classified also for 4-gonal curves of genus $6\le g\le 10$ see section \ref{sec:4gonal}. Remark that, since general curves of genus 5 are complete intersections of three quadrics, their second syzygy schemes coincide always with the whole projective space.

In the trigonal case the situation is as follows. $\mathrm{Syz}_2(C)=\mathbb P^3$ if $g=4$ because $\dim(I_{C,2})=1$.  If $C$ is trigonal of genus $g \ge 5$ then $I_{C,2}=I_{F,2}$ where $F$ is the ruled surface containing $C$ swept by the trisecant lines to $C$. Then $\mathrm{Syz}_2(C)= \mathrm{Syz}_2(F)=F$ because $F$ is determinantal.

Note that canonical curves of gonality at least 4 have all the same numbers of linear syzygies among the quadrics, see for example Lemma \ref{L:k21can}, and hence the $\mathrm{Syz}_2(C)$ truly captures the geometry of $C$ in a more refined way, not being a simple question of number of syzygies.

For the proof of the main result, first we note that if a canonical curve is of Clifford index $\ge 3$ then $\mathrm{Syz}_2(C)=C$, Proposition \ref{P:n2syz}. Hence, it suffices to reduce to the case of curves of Clifford index 2. The second syzygy scheme of a plane sextic is a Veronese surface, Section \ref{P:n2syz}. Therefore, we reduce the analysis to the case of 4-gonal curves. 
The proof in this case is based on the classical approach using scrolls associated to pencils~\cite{fS86} and the sketch is the following.  A canonical 4--gonal curve defines a 3--dimensional scroll $X$ in $\mathbb P^{g-1}$ \cite{fS86} with bounded scrollar invariants, Section \ref{sec:4gonal}. The curve $C$ is a complete intersection of two quadric sections in $X$, \cite[Corollary 4.4]{fS86} and Section \ref{sec:4gonal}. Depending on the numerical type of these two quadrics, their liftings to $\mathbb P^{g-1}$ either are related with the other quadrics in $I_{X,2}$ by a linear syzygy, case in which $\mathrm{Syz}_2(C)=C$, or we are not in this situation, and then $\mathrm{Syz}_2(C)$ is a surface of degree $g-1$. The proof is completed by the classification of surfaces of almost minimal degree. The difference between the two cases follows from Section \ref{sec:scrolls}. The method that we use to produce quadrics and syzygies goes back to K. Petri,  \cite[\S 5]{kP22}. It has been given a modern abstract form by F.--O. Schreyer in \cite{fS86}, see also \cite{BH15}, \cite{mH16},  \cite{cB15}.

\medskip

{\bf Notation.} For any linearly normal variety $X\subset\mathbb P^N$ over the complex numbers, we use the following notation: $K_{p,q}(X)$ for $p,q\ge 0$ are the Koszul cohomology spaces of the polarized variety $(X,\mathcal O_X(1))$. For the definition and basic properties of Koszul cohomology we refer to \cite{mG84} and \cite{AN10}. The dimensions of these spaces are denoted by $\kappa_{pq}(X)$.

%--- I remember vaguely a paper by von Bothmer treating 4-gonal curves from this pt of view.

%%%%%%%%%%%%%%%%%%%%%%%%%%%%%%%%%%%%%%%%%%%%%%%%%%%

\section{The second syzygy scheme}

Let  $X\subset \bP^N$ be a complex projective variety of codimension $\ge 2$   satisfying property $(N_1)$, i.e. projectively normal and such that the homogeneous ideal $I_X\subset \bC[X_0,\dots,X_N]$ is generated by $I_{X,2}$.  We recall the definition of the second syzygy scheme $\Syz_2(X)$. 
If 
\[
0 \ne \gamma \in K_{2,1}(X) \cong K_{1,2}(I_X) = \ker\{V\otimes I_{X,2} \to I_{X,3}\}
\]
then, writing $\gamma=\sum \ell_i\otimes Q_i$, $\ell_i\in V$, $Q_i\in I_{X,2}$, we define:
\[
\Syz(\gamma) = \bigcap_i V(Q_i)
\]
and
\[
\Syz_2(X) = \bigcap_{0\ne\gamma\in K_{2,1}(X)} \Syz(\gamma)
\]

\begin{prop}
\label{P:n2syz}   
Let  $X\subset \bP^N$ be  projectively normal of codimension $\ge 2$ and satisfying property $(N_2)$.  Then 
\[
\Syz_2(X)=X
\]
\end{prop}

\begin{proof}
Since $X$ is an intersection of quadrics (the property $(N_2)$ implies also $(N_1)$), it suffices to show that every $Q\in I_{X,2}$  appears in some  linear syzygy $\gamma$.   But   $\dim(I_{X,2}) \ge 2$ and therefore every $Q\in I_{X,2}$ must appear in \emph{some}  syzygy among quadrics (possibly a trivial one).     
     Property $(N_2)$ implies all the syzygies among the quadrics of $I_{X,2}$ are combination of the linear ones, and therefore $Q$ appears in some \emph{linear} syzygy $\gamma$.  Let us give some details on this claim. Choose a basis $Q_1,\ldots,Q_m$ in $I_{X,2}$ such that $Q=Q_1$.  A basis for the linear syzygies is formed by rows of linear foms $R_i=(\ell_{i1},\ldots,\ell_{im})$ with $\sum\ell_{ij}Q_j=0$ for all $i$. If $Q_1$ does not appear in any linear syzygy, then $\ell_{i1}=0$ for all $i$. The row $(-Q_2,Q_1,0,\ldots,0)$ represents a quadratic syzygy between $Q_1$ and $Q_2$ and hence it is a combination of rows $R_i$ with coefficients linear forms. But the first element of any combination of $R_i$ is zero, which represents a contradiction to $Q_2\ne 0$. 
\end{proof}

Proposition \ref{P:n2syz} applies in particular to canonical curves of Cliff$(C)\ge 3$ and genus $g \ge 6$. It is natural to ask to characterize the canonical curves  satisfying $(N_1)$ and such that $\Syz_2(C)=C$.

Obviously, thanks to Proposition  \ref{P:n2syz}, we only have to look among the  curves $C$ that do not satisfy $(N_2)$. Recall that any such curve must be either a nonsingular plane sextic ($g=10$) or 4-gonal (Green's conjecture in this case). Observe also that a general canonical curve of genus 5 is a complete intersection of three quadrics, and thus $\kappa_{21}(C)=0$; hence $\Syz_2(C)=\bP^4 \ne C$. Therefore we may assume $g \ge 6$ from now on.

In what follows we will   need the
 
\begin{lemma}\label{L:k21can}
Let $C\subset \bP^{g-1}$ be a  canonical curve of genus $g \ge 5$, such that $\mathrm{Cliff}(C)\ge 2$. Then we have:
\begin{equation}\label{k21can}
\kappa_{21}(C) = \frac{(g-1)(g-3)(g-5)}{3}
\end{equation}
\end{lemma}

\begin{proof}
It follows from an  easy computation based on the exact sequence of vector spaces:
\[
0 \to K_{2,1}(C) \to H^0(K)\otimes I_{C,2} \to I_{C,3} \to 0
\]
\end{proof}
 Before attacking the problem in general we will consider a few special cases.

 %%%%%%%%%%%%%%%%%%%%%%%%%%%%%%%%%%%%%%%%%%

\section{Canonical curves of genus 6}
\label{sec:genus6}

First recall a few classical well known facts. Let $C \subset \bP^5$ be a general canonical curve of genus 6. The canonical sheaf $\omega_C=\cO_C(K)$ decomposes in five ways as $\omega_C=\cO(D)+\cO(K-D)$ where $|D|$ is a $g^1_4$ and $|K-D|$ is a $g^2_6$. Each $g^2_6$ maps $C$ birationally to a plane irreducible sextic  $\Gamma \subset \bP^2$ having four nodes $P_1, \dots, P_4$. The residual $g^1_4$ $|D|$ is cut by the pencil of conics through $\{P_1, \dots, P_4\}$.  The other four $g^1_4$ are cut by the four pencils of lines through each $P_i$. 

The linear system $|\I_{\{P_1, \dots, P_4\}}(3)|$ of plane cubics containing $P_1, \dots, P_4$ is adjoint to $\Gamma$ and defines $\varphi: \bP^2\da \bP^5$ which maps $\Gamma$ birationally to $C\subset \bP^5$. The blow-up of $\bP^2$ at $P_1, \dots, P_4$ is mapped isomorphically to a Del Pezzo surface $S \subset \bP^5$ containing $C$.  This is the unique Del Pezzo surface containing $C$. In other words all the five $g^2_6$ produce the \emph{same} surface $S$.

Each $|D|$ is cut on $C$ by a one-dimensional family of 4-secant planes whose union is a cubic 3-fold $X$ of minimal degree containing $S$. There are five such 3-folds $X_1, \dots, X_5$ , and 
\[
S = X_1 \cap \dots \cap X_5
\]

In terms of equations we can describe the above picture as follows. Since
\[
\kappa_{11}(S) = h^0(\cO_{\bP^5}(2))- h^0(\bP^2,\I_{2P_1,\dots, 2P_4}(6))= 21-16=5
\]
the ideal $I_S$ is generated by 5 quadrics $Q_1, \dots , Q_5$. They are the pfaffians of a skew--symmetric matrix $A$ of linear forms, and their syzygies are generated by the 5 columns of $A$ 
%(see Mukai ??). 
Therefore 
\[
\kappa_{21}(S)=5
\]
and
\[
\Syz_2(S)=S
\]
because each $Q_i$ appears in some syzygy.

Since 
\[
\kappa_{11}(C) = h^0(\I_C(2)) = \binom{4}{2} = 6
\]
 $I_C$ is generated by 6 quadrics, that can be taken to be $Q_1, \dots , Q_5,Q$ for some $Q$. Therefore $C = S \cap Q$.  The quadric $Q$ is a general one, and therefore it only has the trivial syzygy with each $Q_i$, and this is not generated by the syzygies among the $Q_i$'s. Therefore 
\[
\kappa_{22}(C)=5
\]
Recall that, according to \eqref{k21can}: 
\[
\kappa_{21}(C)=5
\]
and  $K_{21}(S) \subset K_{21}(C)$.  Therefore they are equal and
\[
\Syz_2(C)= \Syz_2(S) = S \ne C
\]
The output of this discussion is that \emph{$C$ satisfies $(N_1)$ but $\Syz_2(C) \ne C$.}  Therefore  in Problem 2 we may assume $g \ge 7$.

The above numerical informations are collected in the  Betti tables of $C$ and $S$.

\begin{center}
$C$:\quad\begin{tabular}{cccccc}
  %\hline
  $\kappa_{0q}$:&1 & -- & -- & -- & -- \\
  %\hline
  $\kappa_{1q}$:& -- & 6 & 5 & -- & -- \\
  %\hline
  $\kappa_{2q}$:& -- & -- & 5 & 6 & -- \\
  %\hline
  $\kappa_{3q}$:& -- & -- & -- & -- & 1 \\
  %\hline
\end{tabular}
 \quad
$S$:\quad\begin{tabular}{cccc}
1&--&--&-- \\
--&5&5&-- \\
--&--&--&1
\end{tabular}
\end{center}

%%%%%%%%%%%%%%%%%%%%%%%%%%%%%%%%%%%%%%%%%%

\section{Plane sextics}
\label{sec:sextics}

A nonsingular plane sextic $\Gamma \subset \bP^2$ has genus 10. The linear system of cubics defines a morphism:
\[
\varphi_3: \bP^2 \to \bP^9
\]
whose image is the Veronese surface $S := \varphi_3(\bP^2)$. Under this morphism $\Gamma$ is mapped to a canonical curve $C \subset S$.   We have:
\[
\kappa_{11}(C) = \binom{8}{2} = 28,  \quad \kappa_{21}(C) = 105
\]
Moreover:
\[
\kappa_{11}(S) = \dim(I_{S,2}) = \binom{9+2}{2}- h^0(\bP^2, \cO(6)) = 27
\]
From these numbers it is apparent that $C$ does not satisfy $(N_2)$ because it is the complete intersection of $S$ with a quadric $Q$:    therefore there are plenty of trivial relations (of degree 2) between $Q$ and the quadrics in $I_{S,2}$ which are not   combination of linear ones.

Further, since $S$ is projectively normal  we have the exact sequence:
\[
0\to K_{2,1}(S) \to V\otimes I_{S,2} \to I_{S,3}\to 0
\]
where $V = H^0(S,\cO_S(1))$ is 10-dimensional.  Computing we get
\[
\dim(I_{S,3}) = \binom{12}{3}- h^0(\bP^2,\cO(9)) = 220 - 55 = 165
\]
and therefore:
\[
\kappa_{21}(S) = 10 \cdot 27 - 165 = 105
\]
Therefore $K_{2,1}(C)= K_{2,1}(S)$. In particular 
\[
C \ne S \subset \Syz_2(S)=\Syz_2(C)
\]
In fact, $S = \Syz_2(S)$ which can by proved by reducing to a hyperplane section of $S$, which is a normal elliptic curve in $\bP^8$.  Therefore. in Problem 2 we may exclude plane sextics.
  
%  \medskip
%  
%  \emph{Note:}  We did not prove that $S = \Syz_2(S)$ but this is seen to be true by reducing to a hyperplane section of $S$, which is a normal elliptic curve in $\bP^8$ (see next section).

%%%%%%%%%%%%%%%%%%%%%%%%%%%%%%%%%%%%%%%%

\section{Canonical curves either lying on a Del Pezzo surface or bielliptic}

Bielliptic canonical curves, respectively canonical curves lying on a Del Pezzo surface,  are complete intersection of  a quadric with a cone $S$ over a normal curve of degree $g-1$ in $\bP^{g-2}$, respectively  with a Del Pezzo surface $S$.   This case includes both the genus 6 case and the plane sextic case ($g=10$).
Since Del Pezzo surfaces exist only in $\bP^N$ for $3 \le N \le 9$ this case adds only three new values of $g$, namely $g=7,8,9$.   On the other hand bielliptic curves exist for all $g \ge 4$. The treatment is very similar to the previous ones and it consists in computing the following:
\[
\kappa_{11}(S) = \binom{g-2}{2}-1 = \kappa_{11}(C) -1
\]
\[
\kappa_{21}(S) = \kappa_{21}(C)
\]
Moreover one proves that $\Syz_2(S) = S$, by reducing to a hyperplane section of $S$, which is a normal elliptic curve in $\bP^{g-2}$. Since $C = S \cap Q$ for some quadric $Q$, one deduces that there exist trivial sygygies involving $Q$ which are not combination of linear ones. Therefore $C$ does not satisfy $(N_2)$. Even more, we have 
\[
\Syz_2(C)= \Syz_2(S) = S
\]
and therefore $\Syz_2(C) \ne C$.
We summarize what we have proved so far.

\begin{prop}
\label{P:excs}
Let $C\subset \bP^{g-1}$ be a canonical curve of genus $g \ge  6$ satisfying $(N_1)$.  Then $\Syz_2(C) \ne C$ in the following cases:
\begin{itemize}
%\item $g=5$. In this case  $\Syz_2(C)=\bP^4$. 

\item[(1)] $C$ is contained in a Del Pezzo surface. In this case  $5 \le g \le 10$. This case includes all generic  curves of genus 5 and 6, and nonsingular plane sextics ($g=10$) and we have $\Syz_2(C)=S$ 
%except for the case $g=5$ (see above).

\item[(2)] $C$ is bielliptic.

\end{itemize}
\end{prop}

In the next two sections, we prove that the converse also holds.

%%%%%%%%%%%%%%%%%%%%%%%%%%%%%%%%%%%%%%

\section{Curves in 3-dimensional scrolls}
\label{sec:scrolls}

We denote by $P= \bC[Z_1, \dots, Z_g]$ the polynomial ring in $g$ variables. 
Assume $g \ge 6$ and let $X \subset \bP^{g-1}$ be a rational normal scroll of dimension 3 and minimal degree $g-3$. Then 
\[
X = \bP(\mathcal{E})
\]
where   
$$\mathcal{E}:= \cO_{\bP^1}(k_1)\oplus \cO_{\bP^1}(k_2)\oplus \cO_{\bP^1}(k_3)
$$
with
\begin{equation}\label{E:kappas1}
k_1+k_2+k_3 = g-3 = \deg(X), \quad 0\le k_1 \le k_2 \le k_3
\end{equation}
We denote by $H$ a hyperplane section of $X$ and by $F\cong \bP^2$ a fibre of the projection $\pi:X \to \bP^1$. Then $\Pic(X)= \bZ H\oplus \bZ F$ and we have:
\[
H^3=g-3, \quad H^2F=1, \quad HF^2=F^3=0
\]
The hyperplane line bundle $\cO_X(H)$ coincides with the tautological line bundle $\cO_X(1)$. Moreover a canonical divisor is given by:
\[
K_X = -3H+(g-5)F
\]
The scroll $X$ is arithmetically Cohen-Macaulay (aCM) with curve sections rational normal of degree $g-3$ in $\bP^{g-3}$.  Therefore the graded ideal $I_X\subset P$   is generated by $\binom{g-3}{2}$  quadrics, which can be described in a determinantal way as follows. Given $k \ge 1$ and $0\le \alpha\le g-k-1$ we let $M_k(\alpha)$ be the $2\times k$ array:
\[
M_k(\alpha) = \left(
\begin{matrix} Z_{\alpha+1}&\cdots&Z_{\alpha+k} \\ 
Z_{\alpha+2}&\cdots&Z_{\alpha+k+1} 
\end{matrix}
\right)
\]
Then, assuming  that  $k_1 \ge 1$, we can consider the $2\times (g-3)$-matrix of indeterminates:
\begin{align}
M:= &\begin{pmatrix} M_{k_1}(0) & M_{k_2}(k_1+1)& M_{k_3}(k_1+k_2+2) \end{pmatrix} \notag\\
&=\begin{pmatrix}Z_{1}&\cdots&Z_{k_1} &Z_{k_1+2}&\cdots&Z_{k_1+k_2+1}&Z_{k_1+k_2+3}&\cdots&Z_{g-1}\\
Z_2&\cdots&Z_{k_1+1}&Z_{k_1+3}&\cdots&Z_{k_1+k_2+2}&Z_{k_1+k_2+4}&\cdots&Z_g \end{pmatrix}
\notag\end{align}
The $2\times 2$ minors of $M$ define $\binom{g-3}{2}$ linearly independent quadrics which generate $I_X$.

If $k_1=0$ but $k_2 \ge 1$ we obtain a similar $2\times (g-3)$ matrix consisting of only two arrays and involving the variables $Z_1, \dots, Z_{g-1}$.  In this case the quadrics obtained do not involve the variable $Z_1$ and therefore $X$ is a cone with vertex $(1:0:\cdots:0)$.

All this is very explicit and well known. We have:
\[
h^0(X,\cO_X(2H)) = \binom{g+1}{2}- \binom{g-3}{2} = 4g-6
\]
Moreover 
\begin{equation}\label{E:quadrX}
h^0(X,\cO_X(2H-\lambda F)) = h^0(\bP^1,S^2\mathcal{E}(-\lambda)) \ge 4g-6(\lambda+1)
\end{equation}
because $S^2\mathcal{E}$ has rank six.
Observe that this number is positive if  
\begin{equation}\label{E:lambda1}
3\lambda \le 2g-4
\end{equation} 
The inequality \eqref{E:quadrX} can be strict even when $\lambda$ satisfies \eqref{E:lambda1}.

\begin{lemma}\label{L:syz1}
Let $p_1,p_2\in \bP^1$ be distinct points and let $F_i=\pi^{-1}(p_i)$. Let $Q_1\in H^0(X,\cO_X(2H-F_1))$. Then there is $Q_2\in H^0(X,\cO_X(2H-F_2))$  such that any two liftings 
$\wt Q_1,\wt Q_2\in H^0(\bP^{g-1},\cO_{\bP^{g-1}}(2))$   of $Q_1$ and $Q_2$   are related by a linear syzygy and have the same vanishing locus away from fibre components.
\end{lemma}

\begin{proof}
The proof  is  related to the \emph{procedure of rolling factors}  \cite[pages 95--96]{jS03}, attributed to M. Reid. 

We may assume that   $F_1$   and $F_2$ are   the 2-planes   whose equations in $\bP^{g-1}$ are the entries of the first row and of the second row respectively of the matrix $M$. For simplicity let's represent $M$ in the following form:
$$
\begin{pmatrix}
Y_1&Y_2&\dots&Y_{g-3}\\W_1&W_2&\cdots&W_{g-3}
\end{pmatrix}
$$
where the $Y_j$'s and the $W_j$'s are linear forms. Set:
$$
M_{jk}:= Y_jW_k-Y_kW_j
$$
for all $1 \le j < k \le g-3$.
The hypothesis on $Q_1$ then implies that
$$
\wt Q_1 =  \sum_j A_j Y_j
$$
where the $A_j$'s   are linear forms. Take:
$$
\wt Q_2 := \sum_j A_j W_j
$$
and consider the following linear forms:
$$
\wt H_1= \sum_k \alpha_k W_k, \quad \wt H_2 = \sum_k \alpha_k Y_k
$$
where the $\alpha_k\in \mathbb{C}$ are general.  Then:
$$
\wt H_2\wt Q_1 - \wt H_1\wt Q_2 = \sum_{j<k} \Delta_{jk}M_{jk}
$$
where $\Delta_{jk}= A_j\alpha_k-A_k\alpha_j$. 
This is a linear sygyzy  involving  $\wt Q_1,\wt Q_2$. The last assertion is  obvious.
\end{proof}

\begin{prop}\label{P:canonX}
 Consider non-negative integers $a,b$ such that $a+b=g-5$ and
$$
H^0(X,\cO_X(2H-aF))\ne 0 \ne H^0(X,\cO_X(2H-bF))
$$
and assume that    
\[
C := Q_1 \cap Q_2 \subset X
\]
is a nonsingular curve for some $Q_1\in |2H-aF|$,  $Q_2\in |2H-bF|$. Then \begin{itemize}
\item[(i)] 
$C$ is a 4-gonal canonical curve of genus $g$.  
\item[(ii)] A basis of $I_{C,2}$ is given by the $\binom{g-3}{2}$  minors of the matrix $M$ plus   $g-3$ quadrics which are liftings of a basis of the subspace of $H^0(X,\cO_X(2H))$  generated by $\pi^*H^0(\bP^1,\cO(a))\cdot Q_1$ and $\pi^*H^0(\bP^1,\cO(b))\cdot Q_2$.
\end{itemize}
\end{prop}

\begin{proof} (i) follows from a straightforward calculation.

(ii) The subspaces $\pi^*H^0(\bP^1,\cO(a))\cdot Q_1$ and $\pi^*H^0(\bP^1,\cO(b))\cdot Q_2$ have dimensions $a+1$ and $b+1$ respectively, and have zero intersection because otherwise, 
%by Lemma \ref{L:syz1}, 
$Q_1$ and $Q_2$ would not intersect transversally.  Therefore the space they generate has dimension $g-3=(a+1)+(b+1)$. Since
$$
\binom{g-3}{2}+(g-3) = \binom{g-2}{2} = \dim(I_{C,2})
$$
we are done.
\end{proof}

\begin{corollary}\label{C:canonX}
Let $C\subset \bP^{g-1}$ be as in the Proposition and assume that both $a$ and $b$ are positive. Then $\Syz_2(C)=C$.
\end{corollary}

\begin{proof}
The condition $a > 0$ implies, by Lemma \ref{L:syz1}, that we can choose a basis of 
$\pi^*H^0(\bP^1,\cO(a))\cdot Q_1$ whose liftings in $I_{C,2}$ are related in pairs by a linear syzygy. Similarly for   $\pi^*H^0(\bP^1,\cO(b))\cdot Q_2$.  Since   the minors of $M$ are also related by linear syzygies, we have $\Syz_2(C)=C$.
\end{proof}

\begin{remark}\rm
Let $Q_1$ and $Q_2$ be as in Proposition \ref{P:canonX}. Then for any liftings $\wt Q_1,\wt Q_2\in H^0(\bP^{g-1},\cO_{\bP^{g-1}}(2))$, of $Q_1$, $Q_2$, respectively, the trivial syzygy 
$$
\wt Q_2\cdot \wt Q_1 - \wt Q_1 \cdot \wt Q_2 =0
$$
is not a combination of linear ones and therefore it is responsible for the failure of $(N_2)$ on $C$. Obviously, it suffices to check this fact modulo $I_X$. Hence, we will reduce everything modulo $I_X$ and we replace $I_C$ by $I_{C/X}$.
Recall that the liftings of $Q_1$ are parameterized by $\pi^*H^0(\bP^1,\cO(a))\cdot Q_1$ and the liftings of $Q_2$ are parameterized by $\pi^*H^0(\bP^1,\cO(b))\cdot Q_2$ and their classes modulo $I_X$ generate together the $(g-3)$-dimensional subspace   
$$
I_{C/X,2}=[\pi^*H^0(\bP^1,\cO(a))\cdot Q_1]\oplus [\pi^*H^0(\bP^1,\cO(b))\cdot Q_2]
$$
of $H^0(X,\cO_X(2H))$.   Choose bases $Q_{01},\ldots,Q_{a1}\in \pi^*H^0(\bP^1,\cO(a))\cdot Q_1$, respectively, $Q_{02},\ldots,Q_{b2}\in \pi^*H^0(\bP^1,\cO(b))\cdot Q_2$ such that $Q_{01}=\wt Q_1\mbox{ mod }I_X$ and $Q_{02}=\wt Q_2\mbox{ mod }I_X$.

An easy computation gives $h^0(X,\cO_X(3H))= 10g-20$ and therefore:
$$
\dim(I_{C/X,3})= h^0(X,\cO_X(3H))-h^0(C,\cO_C(3K))=5(g-3)
$$
The space of linear sygygies on $I_{C/X,2}$ has dimension:
$$
g\cdot \dim(I_{C/X,2})- \dim(I_{C/X,3}) = g(g-3)-5(g-3)= (g-5)(g-3)
$$
On the other hand the linear syzygies involving only $Q_{01},\ldots,Q_{a1}$ are those in the kernel of
$$
H^0(X,\cO_X(H))\times \pi^*H^0(\bP^1,\cO(a)) \to H^0(X,\cO_X(H+aF))
$$
Since this map is surjective (easy to check) its kernel has dimension 
$$
g(a+1)-(g+3a)= a(g-3)
$$
Similarly the kernel of the map giving the syzygies involving only $Q_{02},\ldots,Q_{b2}$ has dimension $b(g-3)$.
We get altogether a space of linear syzygies of dimension $(a+b)(g-3)=(g-5)(g-3)$.  Comparing with the previous computation we see that these account for \emph{all} the linear syzygies among $I_{C/X,2}$.  In particular, any linear syzygy 
\[
\sum L_{i1}Q_{i1}+\sum L_{j2}Q_{j2}=0
\]
will split into two linear syzygies
\[
\sum L_{i1}Q_{i1}=0,\ \sum L_{j2}Q_{j2}=0.
\]
This implies in particular that no combination of linear syzygies can produce the quadratic syzygy $Q_{01}Q_{02}-Q_{02}Q_{01}=0$. Indeed, if
\[
(Q_{02},0,\ldots,0,-Q_{01},0,\ldots,0)=\sum\ell_kR_k
\]
with 
\[
R_k=(L^{(k)}_{01},\ldots,L^{(k)}_{a1},L^{(k)}_{02},\ldots,L^{(k)}_{b2})
\]
such that $\sum L^{(k)}_{i1}Q_{i1}+\sum L^{(k)}_{j2}Q_{j2}=0$, it will give separate linear syzygies
$\sum L^{(k)}_{i1}Q_{i1}=0$ and $\sum L^{(k)}_{j2}Q_{j2}=0$ which eventually implies $Q_{01}Q_{02}=0$, contradiction.
\end{remark}

%%%%%%%%%%%%%%%%%%%%%%%%%%%%%%%%%%%%%%%%%%%%

\section{The second syzygy scheme of a 4-gonal canonical  curve}
\label{sec:4gonal}

Consider $C\subset\mathbb P^{g-1}$ a canonical 4--gonal curve of genus $g$ and choose $L$ a base-point-free $g^1_4$ on $C$. The scroll defined by $L$ is by definition \cite{fS86}
\[
X:=\bigcup_{D\in|L|}\langle D\rangle.
\]
Let us write $X=\mathbb P(\mathcal O_{\mathbb P^1}(k_1)\oplus\mathcal O_{\mathbb P^1}(k_2)\oplus\mathcal O_{\mathbb P^1}(k_3))$ with $0\le k_1\le k_2\le k_3$ and $k_1+k_2+k_3=g-3$ and denote as usual $H$ a hyperplane section of $X$ and $F$ a fibre over a point $p\in\mathbb P^1$. Note that, since 
$$
h^0(X,\mathcal O_X(H-\lambda F))=h^0(\mathbb P^1,\mathcal O_{\mathbb P^1}(k_1-\lambda))+h^0(\mathbb P^1,\mathcal O_{\mathbb P^1}(k_1-\lambda))+h^0(\mathbb P^1,\mathcal O_{\mathbb P^1}(k_1-\lambda))
$$
for all $\lambda$, we have the following interpretation of the highest scrollar invariant 
\[
k_3=\mathrm{max}\{\lambda:\ h^0(X,\mathcal O_X(H-\lambda F))>0\}.
\]
In particular, since $F\cdot C=4$, and $\mathcal O_C(H)=K_C$ we obtain $4k_3\le 2g-2$ i.e. $k_3\le \left[\frac{g-1}{2}\right]$. One consequence of this fact is that if $g\ge 6$ then  $X$ is either smooth ($k_1\ge 1$) or it has one single singular point ($k_1=0$ and $k_2\ge 1$).

\begin{prop}
\label{prop:ScrollBielliptic}
Let us assume $C$ is bielliptic of genus $g\ge 6$. Then the scroll associated  to any $g^1_4$ on $C$ is singular and the vertex coincides with the vertex of the elliptic cone containing~$C$.
\end{prop}

\proof
The situation is the following. Let $f:C\to E$ be the double cover. Then $f_*(K_C)\cong\mathcal O_E\oplus\mathcal O_E(B)$ with $B$ the branch divisor. It follows that $H^0(K_C)=\mathbb C\oplus H^0(\mathcal O_E(B))$ and hence the above splitting induces a natural cone containing $C$ whose hyperplane section is $E$ embedded by $|B|$ into a hyperplane $\mathbb P^{g-2}$.
Then, for any $g^1_4$ on $C$ the associated 3-dimensional scroll $X$ in $\mathbb P^{g-1}$ defined as the union of the corresponding $4$--secant $2$--planes is singular at one single point which is precisely the vertex $v$ of the elliptic cone. Indeed, let $z_1+z_2$ be a $2$--cycle on $E$ and $f^{-1}(z_1)=x_1+y_1$, $f^{-1}(z_2)=x_2+y_2$. The fibre of $X$ is by definition the plane spanned by $x_1,x_2,y_1,y_2$ and hence contains the lines $\overline{x_1y_1}=\overline{vz_1}$ and $\overline{x_2y_2}=\overline{vz_2}$ which pass through the vertex of the elliptic cone. In particular, all the fibres of the 3-dimensional scroll pass through the vertex of the scroll and hence $v$ is a singular point. 

The fact  that $v$ is the only singular point of $X$ can be also proved geometrically in this case as follows. Assume that all the fibres contain a whole line $L$.  Since $L$ passes through $v$ which is not contained in the hyperplane $\mathbb P^{g-2}=|B|$, $L$ intersects this hyperplane in one point $u$.  Then the line $\overline{z_1z_2}$ which is contained in the plane $\langle x_1,x_2,y_1,y_2\rangle$ must interesect $L$ and hence passes through $u$. Hence for any other two points $z_2,z_3$ on $E$ with $z_1+z_2\sim z_3+z_4$ the line $\overline{z_3z_4}$ also passes through $u$ which implies that $\langle z_1,z_2,z_3,z_4\rangle\subset \mathbb P^{g-2}$ is a $4$--secant $2$--plane for $E$, hence $z_1+z_2+z_3+z_4$ fails to impose independent conditions on $|B|$, in particular $B-z_1-z_2-z_3-z_4$ is special. This is possible only if $\mathcal O_E(B)$ is of degree 4, i.e. $g=5$, contradiction.
\endproof

The main aim of this section is to prove the following result, which answers Question 3:

\begin{theorem}
\label{thm:main}
Let $C\subset X \subset \bP^{g-1}$ be a 4-gonal canonical curve of genus $g \ge 6$.  Then $\Syz_2(C)= C$ if and only if $C$ is neither bielliptic nor it lies on a Del Pezzo surface. 
\end{theorem}

\begin{proof}
 Let $C\subset X\subset \bP^{g-1}$ where $X$ is the scroll determined by $|D|$, a complete $g^1_4$ on $C$. We have:
\begin{equation}\label{E:quadrCX}
h^0(X,\mathcal{I}_{C/X}(2H))= \binom{g-2}{2}- \binom{g-3}{2}=g-3
\end{equation}
Let  
$$
\lambda_0 =\mathrm{max}\left\lbrace\lambda: h^0(X,\mathcal{I}_{C/X}(2H-\lambda F))> 0   \right\rbrace
$$
Then $\lambda_0 \le g-5$. In fact choose $0\ne Q_1 \in H^0(X,\mathcal{I}_{C/X}(2H-\lambda_0 F))$; then $H^0(\bP^1,\cO_{\bP^1}(\lambda_0))\cdot Q_1\subset H^0(X,\mathcal{I}_{C/X}(2H))$ must be a subspace of dimension $\lambda_0+1\le g-4$, otherwise the $g-3$ quadrics   in $H^0(X,\mathcal{I}_{C/X}(2H))$ intersect in a surface. Whence the stated inequality follows. 
Now   take $\lambda_1= g-5-\lambda_0$ and 
$$
Q_2\in H^0(X,\mathcal{I}_{C/X}(2H-\lambda_1 F))\setminus H^0(\bP^1,\cO_{\bP^1}(\lambda_0-\lambda_1))\cdot Q_1.
$$

%\marginpar{\tiny We have to prove this formula on $\pi_*\mathcal{I}_{C/X}(2H)$. Apply Grothendieck-Riemann-Roch to find its determinant maybe?} 
It is easy to see that such a quadric exists, since $\pi_*\mathcal{I}_{C/X}(2H)=\mathcal O_{\bP^1}(\lambda_0)\oplus\mathcal O_{\bP^1}(\lambda_1)$ from the definition and hence $h^0(X,\mathcal{I}_{C/X}(2H-\lambda_1 F))=1+h^0(\bP^1,\cO_{\bP^1}(\lambda_0-\lambda_1))$. Another argument for the existence of $Q_2$ is provided by \cite[Corollary 4.4]{fS86} which shows that $C$ is a complete intersection of two quadric sections in $X$ one of which is $Q_1$, see also \cite[(6.3)]{fS86}. Note that the inequalities $0\le \lambda_1\le\lambda_0\le g-5$ are also proved in \cite[(6.2)--(6.3)]{fS86}.

Then  $H^0(X,\mathcal{I}_{C/X}(2H))$ is generated by $H^0(\bP^1,\cO_{\bP^1}(\lambda_0))\cdot Q_1$ and 
$H^0(\bP^1,\cO_{\bP^1}(\lambda_0))\cdot Q_1$  (\cite[Corollary 4.4]{fS86} or Proposition \ref{P:canonX}). It follows from Corollary \ref{C:canonX}   that $\Syz_2(C)=C$ unless $\lambda_1=0$.  If $\lambda_1=0$ then $Q_1$ is a surface of degree $(2H-(g-5)F)\cdot H^2= 2(g-3)-(g-5)=g-1$ and hence it is either a Del Pezzo surface or an elliptic cone and $C$ is the complete intersection of $Q_1$ with a  quadric lifting $\wt Q_2$ of $Q_2$.  
\end{proof}

%
%TODO LIST (project):
%\begin{enumerate}
%\item prove that $N_p$ implies $Syz_p(C)=C$ using Koszul relations.
%\item for the tetragonal case, prove that the Petri map is not surjective without using the classification of surfaces of degree $g-1$ (Nagata).
%\item the pentagonal case, corresponding to $Syz_3$. See also Keem. Possible toy example: a curve which is a double cover of a genus 2 curve and has a bpf $g^1_5$. 
%\end{enumerate}

\end{document}